\documentclass{article}

\PassOptionsToPackage{colorlinks=true,linkcolor=red!50!black,citecolor=green!50!black}{hyperref}

\usepackage{adjustbox}
\usepackage{amsfonts}
\usepackage{amssymb}
\usepackage{amsthm}
\usepackage{bm}
\usepackage{booktabs}
\usepackage{enumerate}
\usepackage{fancyhdr}
\usepackage{graphicx}
\usepackage[a4paper]{geometry}
\usepackage[utf8]{inputenc}
\usepackage{lmodern}
\usepackage{mathtools}
\usepackage{multirow}
\usepackage{xcolor}

\usepackage{hyperref}
\usepackage[capitalise]{cleveref}

\expandafter\let\expandafter\oldproof\csname\string\proof\endcsname
\let\oldendproof\endproof

\renewenvironment{proof}[1][\proofname]{
	\oldproof[#1:]\unskip}{\oldendproof\unskip}

\newtheoremstyle{break}
	{}
	{}
	{}
	{}
	{\bfseries}
	{}
	{\newline}
	{}
\newtheoremstyle{definitionbreak}
	{\parskip}
	{\parskip}
	{}
	{}
	{\bfseries}
	{}
	{\newline}
	{}

\theoremstyle{break}
\newtheorem{thm}[subsection]{Theorem}

\newtheorem{lemma}[subsection]{Lemma}

\newtheorem*{remark*}{Remark}
\newtheorem{introthm}{Theorem}

\theoremstyle{definitionbreak}
\newtheorem{defn}[subsection]{Definition}
\newtheorem*{notation*}{Notation}

\let\O\relax

\renewcommand{\epsilon}{\varepsilon}

\newcommand{\nin}{\notin}

\newcommand{\divides}{\bigm|}

\newcommand{\Fq}{\mathbb{F}_q}

\newcommand{\Fgp}{\mathbb{F}_q^*}

\DeclareMathOperator{\Inv}{Inv}
\DeclareMathOperator{\O}{O}
\DeclareMathOperator{\POmega}{P\Omega}
\DeclareMathOperator{\PSL}{PSL}

\DeclareMathOperator{\SL}{SL}
\DeclareMathOperator{\SO}{SO}
\DeclareMathOperator{\SOmega}{\Omega}
\DeclareMathOperator{\Span}{span}

\DeclarePairedDelimiter{\abs}{\lvert}{\rvert}
\DeclarePairedDelimiter{\gen}{\langle}{\rangle}

\begin{document}
\begin{center}
{\Large Maximal Cocliques in \(\PSL_2(q)\)}

{\large Jack Saunders\footnotemark} \footnotetext{This work was done as part of my PhD, thus I would like to thank my supervisor Corneliu Hoffman for everything he has done for me thus far.}

School of Mathematics, University of Birmingham, B15 2TT\vskip-0.8em
JPS675@bham.ac.uk
\end{center}

\begin{abstract}
The generating graph of a finite group is a structure which can be used to encode certain information about the group. It was introduced by Liebeck and Shalev and has been further investigated by Lucchini, Mar\'oti, Roney-Dougal and others. We investigate maximal cocliques (totally disconnected induced subgraphs of the generating graph) in \(\PSL_2(q)\) for \(q\) a prime power and provide a classification of the `large' cocliques when \(q\) is prime. We then provide an interesting geometric example which contradicts this result when \(q\) is not prime and illustrate why the methods used for the prime case do not immediately extend to the prime-power case with the same result.
\end{abstract}


\section{Introduction}

It is a well-known result of Steinberg (and later others) that every finite simple group may be generated by just 2 elements, and from \cite{GenProb,KantorLubotzky} we know that if we pick two elements of the group at random then the chance that they will generate the group is high. This then motivated the definition of the generating graph by \cite{GenGraph}, a graph whose vertex set consists of the non-identity elements of the group \(G\) and we draw an edge between two elements precisely when they generate \(G\).

It is then interesting to look at what information the generating graph alone can tell us about the group. For example, in \cite{LucchiniMarotiRoneyDougal}, Lucchini, Mar\'oti and Roney-Dougal showed that the generating graph determines the group up to isomorphism for sufficiently large simple groups and for symmetric groups. The structure of the graph itself has also been investigated, such as the clique number, the presence of hamiltonian cycles and many other things (some examples: \cite{LucchiniMarotiClique, HamiltonianCycles, LucchiniMarotiQuestions}).

In this case, we will be looking at what we may be able to identify in the group simply by looking at maximal cocliques (totally disconnected induced subgraphs) of the generating graph for \(\PSL_2(q)\) for \(q\) a prime power, and we show that when \(q\) is prime the largest maximal cocliques are either entirely made up of involutions or are maximal subgroups. In particular, we prove the following theorem:

\begin{introthm}
	Let \(G = \PSL_2(p)\) for some prime \(p > 2\) and let \(A\) be a maximal coclique in \(G\). Then \(A\) is either a maximal subgroup, the conjugacy class of all involutions or \(\abs{A} \leq \frac{129}{2} (p-1) + 2\).
\end{introthm}

Further, when \(q\) is no longer prime there is an interesting geometric example which crops up due to the isomorphism \(\PSL_2(q^2) \cong \POmega_4^-(q)\) which we illustrate and show to be a maximal coclique which prevents the extension of the previous theorem to the non-prime case. More formally, we have the following:

\begin{introthm}
	Let \(V\) be a 4-dimensional orthogonal space of \(-\) type over \(\Fq\) and fix some non-isotropic vector \(v \in V\). Then the set of all elements of \(G \coloneqq \POmega_4^-(q)\) with 2-dimensional eigenspaces lying in \(v^{\perp}\) is a maximal coclique of \(G\) of order \(q^3 + q\).
\end{introthm}

We currently believe that, in the prime power case, the example we have found and its conjugates are the only sufficiently large exceptions and that otherwise a similar result to the prime case should hold. The prime result tells us that, in particular, given the generating graph for \(\PSL_2(p)\) we can identify and distinguish the conjugacy class of all involutions and the Borel subgroups by size alone. In the prime power case, it should be possible to identify and distinguish the Borel subgroups, \(\PSL_2(q).2 \leq \PSL_2(q^2)\) and the geometric example.

We proceed by examining the maximal subgroups of \(\PSL_2(p)\) and their intersections in order to provide a linear bound (in \(p\)) on those maximal cocliques which contain at least one element of order greater than 2 and are not maximal subgroups. After this, we provide an example in the case of \(\PSL_2(q^2)\) which does not fit in with the result for \(\PSL_2(p)\) to show that the result does not generalise as-is and provide some evidence to suggest why using the same approach as in the prime case is likely to be either very complicated or impossible.

In addition, we would like to thank the referee for their careful reading of this paper and their very helpful suggestions.

\subsection*{The prime case}

\begin{notation*}
	In what follows I shall use \(\O_n^{\epsilon}(q)\) to denote the corresponding general orthogonal group, \(D_{2n}\) to denote the dihedral group of order \(2n\), \(C_n\) to denote the cyclic group of order \(n\) and \(A_n\) and \(S_n\) to denote the alternating and symmetric groups on \(n\) elements, respectively.  We begin by defining what we will be working with and proving some basic results which will be used throughout.
\end{notation*}

\begin{defn}
	Let \(G\) be a group. Then \(A \subseteq G\) is a {\itshape coclique} if for all \(g, h \in A\) we have that \(\left< g, h \right> \neq G\). We say that a coclique \(A\) is maximal if whenever \(B\) is a coclique we have that \(A \subseteq B \implies A = B\).
\end{defn}

\begin{lemma}	\label{subgp}
	If \(G\) is a group and \(A \subseteq G\) is a maximal coclique then \(\left< A \right> \neq G \implies \left< A \right> = A\) and \(A\) is a maximal subgroup of \(G\).
\end{lemma}

\begin{proof}
Suppose \(\left<A\right> \neq G\). If \(g \in \left<A\right> \setminus A\) then \(\left< A \cup \{ g \} \right> \subseteq \left< A \right> \neq G\). But then this contradicts the maximality of \(A\) and so \(\left< A \right> = A\). Also, if \(A \subseteq H \lneq G\) and we assume \(A \neq H\) then we may pick \(g \in H \setminus A\) and again note that \(\left< A \cup \{g\} \right> \subseteq \left< H \right> = H \neq G\). In particular, \(\left< g, h \right> \neq G\) for all \(h \in A\) but this again contradicts the maximality of \(A\) and so \(A = H\) is a maximal subgroup.
\end{proof}

\begin{lemma}	\label{uniquemax}
	If \(g \in G\) lies in a unique maximal subgroup \(M\) of \(G\) then any maximal coclique \(A\) containing \(g\) must be equal to \(M\).
\end{lemma}

\begin{proof}
Let \(h \nin M\) and consider \(\left< g, h \right>\). Since every proper subgroup is contained in a maximal subgroup, we have that \(\left< g, h \right>\) lies in a maximal subgroup containing \(g\). But \(M\) is the only such subgroup. Thus we must have that \(\left< g, h \right> \subseteq M\) or \(\left< g, h \right> = G\). But \(h \nin M\) and so \(\left< g, h \right> = G\). Thus \(A \subseteq M\) and the maximality of \(A\) gives us that \(A = M\).
\end{proof}

For what we intend to do next, we will need to know the maximal subgroups of \(\PSL_2(p)\) and when they appear. We provide these below, one may find them in  \cite[\textsection260 (p286)]{DicksonLinear} or for a more modern treatment look at the maximal subgroups for \(\SL_2(p)\) in \cite{HoltMaximals} and then simply quotient out by its centre to obtain these subgroups.

\begin{thm}[Maximal subgroups of \(\PSL_2(p)\)]
	The conjugacy classes of maximal subgroups of \(\PSL_2(p)\) are as follows:
	\begin{enumerate}[i)]
	\item The Borel subgroups \(C_p \rtimes C_{\frac{1}{2} (p-1)}\) appear for all \(p\).
	\item \(D_{p-1}\) appears for all \(p > 11\).
	\item \(D_{p+1}\) appears for all \(p > 7\).
	\item \(A_4\) appears when \(p \equiv \pm 3, \, \pm 13 \mod 40\).
	\item \(S_4\) appears when \(p \equiv \pm 1 \mod 8\) and has two conjugacy classes.
	\item \(A_5\) appears when \(p \equiv \pm 1 \mod 10\) and has two conjugacy classes.
	\end{enumerate}
\end{thm}

We shall also make use of Aschbacher's classification of maximal subgroups of classical groups for the prime power case, so we shall include this here as well, though we ignore the classes which do not appear. This is also found in \cite{HoltMaximals}.

\begin{thm}[Aschbacher's Theorem for \(\SL_2(q)\) and \(\SOmega_4^-(q)\)]
	The maximal subgroups of \(\SL_2(q)\) and \(\SOmega_4^-(q)\) (acting naturally on \(V = \Fq^n\)) lie in one of the following classes:

	\begin{tabular}{cl} \toprule
	 {\bfseries Class}& {\bfseries Description}\\ \midrule
	 \(\mathcal{C}_1\) & Stabilisers of totally singular or non-singular subspaces. \\ 
	\(\mathcal{C}_2\) & Stabilisers of direct sum decompositions \(V = \bigoplus_{i=1}^t V_i\) where the \(V_i\) all have the \\
	 & same dimension. \\ 
	\(\mathcal{C}_3\) & Stabilisers of extension fields of \(\Fq\) of index 2. \\ 
	\(\mathcal{C}_5\) & Stabilisers of subfields of \(\Fq\) of prime index. \\ 
	\(\mathcal{C}_6\) & Normalisers of symplectic-type or extraspecial groups in absolutely irreducible \\
	 & representations. \\ 
	\(\mathcal{S}\) & See \cite[Definition 2.1.3]{HoltMaximals}, the specifics of this will not be used here. \\ \bottomrule
	\end{tabular}
\end{thm}

\begin{thm} \label{main}
	Let \(G = \PSL_2(p)\) for some prime \(p > 2\) and let \(A\) be a maximal coclique in \(G\). Then \(A\) is either a maximal subgroup, the conjugacy class of all involutions or \(\abs{A} \leq \frac{129}{2}(p-1) + 2\).
\end{thm}

\begin{remark*}
	In fact, this result splits into two bounds --- we have that \(\abs{A} \leq \frac{93}{2}(p+1)\) when \(p < 7\) and \(\abs{A} \leq \frac{129}{2}(p-1) + 2\) otherwise, though in either case for \(p < 11\) both bounds are larger than \(\abs{G}\) so we may take the latter. Further, the above bounds are not tight, there is definitely room for improvement in the case where we allow for multiple elements of large order lying in distinct cyclic subgroups. However, we can at least show that there exists a maximal coclique of order linear in \(p\) which is not a maximal subgroup or the conjugacy class of all involutions.

	Indeed, in the case where \(G\) contains \(A_4\) as a maximal subgroup, we may fix an element \(x\) of order 3 in \(A_4\) and suppose that \(3 \divides p+1\). Acting by \(N_G(x)\) we see that \(x\) lies in \(\frac{1}{3}(p+1)\) copies of \(A_4\) and so we may include all of the involutions from these subgroups to form a coclique, along with the involutions from \(N_G(x)\). Note that there is no multiple counting here since \(\gen{x}\) is already a maximal subgroup of \(A_4\). We hence obtain a coclique of order \(\frac{3}{2}(p+1) + 3\) (or one higher, if \(N_G(x)\) contains a central involution) which one can check is maximal.

	This is the smallest example we obtain in this way out of all of the cases; the others give larger examples since either we also include contributions from the Borel subgroups or there are more conjugacy classes of these small maximal subgroups in \(G\).
\end{remark*}

We now prove \cref{main} through a series of lemmas. In the following, \(G\) and \(A\) are taken to be as in the statement of the theorem; \(A\) is not a maximal subgroup and contains some \(x\) (\(\abs{x} > 2\)) such that \(A \setminus \left<x\right>\) consists entirely of involutions.

\begin{lemma} \label{main1a}
	If \(G\) and \(A\) are as above with \(\abs{x} \divides p+1\) and \(G\) contains \(A_4\) as a maximal subgroup then \(\abs{A} \leq \frac{3}{2} (p+1) + 4\).
\end{lemma}

\begin{proof}
	Here \(x\) may lie in \(D_{p+1}\) or \(A_4\). If \(\abs{x} > 3\) then \(x\) may only lie in \(D_{p+1}\) and thus lies in a unique maximal subgroup since the intersection of conjugate dihedral groups in \(\PSL_2(p)\) consists only of involutions. Then by \cref{uniquemax} we have that \(A\) must be a maximal subgroup.

	Otherwise \(\abs{x} = 3\) and we must count all of the involutions present in all subgroups in which \(x\) appears. We start off by noting that \(A_4\) contains 3 involutions and a dihedral group \(D_{p+1}\) contains at most \(\frac{p+3}{2}\) involutions (it contains one fewer if \(p \equiv 3 \mod 4\)). In order to figure out how many copies of \(A_4\) can contain \(x\), we note that \(D_{p+1}\) acts on the set of copies of \(A_4\) which contain \(x\) with stabiliser \(N_G(x) \cap A_4 = N_{A_4}(x) = \gen{x}\). This action is transitive since if \(x \in H_1 \cap H_2\) where \(H_1 \cong H_2 \cong A_4\) and \(H_2 = H_1^g\) then there exists \(h \in H_2\) such that \((x^h)^g = x\) and so \(gh \in N_G(x)\) and \(H_1^{gh} = H_2\).

	Thus the orbit of any copy of \(A_4\) containing \(x\) under this action will have length \(\frac{p+1}{3}\). Therefore from the entire conjugacy class of \(A_4\) in \(G\) we shall get a contribution of at most \(3 \frac{p+1}{3}\) involutions. Adding this up, we see that
	\[A \subseteq \left< x \right> \cup \Inv D_{p+1} \cup \frac{p+1}{3} \Inv A_4\]
	where \(\Inv H\) denotes the set of involutions in \(H\), and we abuse notation so that \(k \Inv A_4 := \bigcup_{i=1}^k \Inv G_i\) where \(G_i \cong A_4\) for all \(i\). This gives
	\[\abs{A} \leq 3 + \frac{p+3}{2} + p+1 = \frac{3(p+1)}{2} + 4. \qedhere\]
\end{proof}

\begin{lemma} \label{main1b}
	If \(G\) and \(A\) are as above with \(\abs{x} \divides p+1\) and \(G\) contains \(A_5\) or \(S_4\) as maximal subgroups then \(\abs{A} \leq \frac{17}{2} (p+1) + 4\).
\end{lemma}

\begin{proof}
	Since we are only dealing with upper bounds, we shall combine the cases where either \(A_5\), \(S_4\) or both appear as maximal subgroups of \(G\) as this encompasses the arguments used for each individual case. Our worst case bound occurs when \(\abs{x} = 3\) and both \(A_5\) and \(S_4\) are maximal in \(G\), so we shall only consider this situation. There will be two conjugacy classes of both \(A_5\) and \(S_4\), so we must double the contribution we get from each type of subgroup. We act by \(D_{p+1}\) as in the previous lemma, noting that normalisers in \(A_5\) and \(S_4\) are dihedral rather than cyclic as in \(A_4\), to see that
	\[A \subseteq \left<x\right> \cup \Inv D_{p+1} \cup \frac{2(p+1)}{6} \Inv S_4 \cup \frac{2(p+1)}{6} \Inv A_5\]
	and so
	\[\abs{A} \leq 3 + \frac{p+3}{2} + 3(p+1) + 5(p+1) = \frac{17(p+1)}{2} + 4. \qedhere\]
\end{proof}

\begin{lemma} \label{main2a}
	If \(G\) and \(A\) are as above with \(\abs{x} \divides p-1\) and \(G\) contains \(A_4\) as a maximal subgroup then \(\abs{A} \leq \frac{7}{2}(p-1) + 5\).
\end{lemma}

\begin{proof}
	In this case, we have the most subgroups to consider as \(x\) may lie in a Borel subgroup, \(D_{p-1}\) or \(A_4\).

	We again start by considering the case where \(\abs{x} > 5\). Such an \(x\) will lie in one copy of \(D_{p-1}\) and two distinct Borel subgroups which will consist of the upper and lower triangular matrices with respect to a basis in which \(x\) is diagonal. We need only count the involutions in these subgroups, noting that a Borel subgroup contains at most \(p\) involutions. In this case we get
	\[A \subseteq \left< x \right> \cup \Inv \mathcal{B}_1 \cup \Inv \mathcal{B}_2 \cup \Inv D_{p-1}\]
	where \(\mathcal{B}_1\) and \(\mathcal{B}_2\) are the Borel subgroups in which \(x\) lies. This then gives us
	\[\abs{A} \leq \abs{x} + 2p + \frac{p+1}{2} = \frac{5(p+1)}{2} + \abs{x} - 2.\]
	When \(\abs{x} \leq 5\) we must split this into several subcases as before.

	Since \(G\) contains \(A_4\) as a maximal subgroup we need only concern ourselves with the case where \(\abs{x} = 3\) and we count involutions as before. We may find out how many copies of \(A_4\) contain \(x\) by using the action of \(D_{p-1}\) on the set of \(A_4\)s containing \(x\) in a very similar way to previous cases. So we have 
	\[A \subseteq \left< x \right> \cup \Inv \mathcal{B}_1 \cup \Inv \mathcal{B}_2 \cup \Inv D_{p-1} \cup \frac{p-1}{3} \Inv A_4\]
	which then gives
	\[\abs{A} \leq 3 + 2p + \frac{p-1}{2} + p-1 = \frac{7(p-1)}{2} + 5. \qedhere\]
\end{proof}

\begin{lemma} \label{main2b}
	If \(G\) and \(A\) are as above with \(\abs{x} \divides p-1\) and \(G\) contains \(A_5\) or \(S_4\) as maximal subgroups then \(\abs{A} \leq \frac{21}{2}(p-1) + 5\).
\end{lemma}

\begin{proof}
	As with the last time, we shall merge all of the cases where \(A_5\), \(S_4\) or both appear as maximal subgroups of \(G\). Acting with \(D_{p-1}\) we see that
	\[A \subseteq \left<x\right> \cup \Inv \mathcal{B}_1 \cup \Inv \mathcal{B}_2 \cup \Inv D_{p-1} \cup \frac{2(p-1)}{6} \Inv A_5 \cup \frac{2(p-1)}{6} \Inv S_4\]
	and so, noting that if \(D_{p-1}\) contains a central involution it will also lie in a Borel subgroup,
	\[\abs{A} \leq 3 + 2p + \frac{p-1}{2} + 5(p-1) + 3(p-1) = \frac{21(p-1)}{2} + 5. \qedhere\]
\end{proof}

We now remove the condition that \(A\) contains essentially only one element of large order and simply suppose it has at least one element of order greater than 2.

\begin{lemma} \label{main3a}
	Choose some \(x \in A\) with \(\abs{x} > 2\). If \(\abs{x} \divides p+1\) and \(G\) contains \(A_4\) as a maximal subgroup then \(\abs{A} \leq \frac{9}{2} (p+1) + 1\).
\end{lemma}

\begin{proof}
	If \(\abs{x} > 3\) then \(x\) may only lie in \(D_{p+1}\) and thus by \cref{uniquemax} we are done. Otherwise, we need only concern ourselves with \(\abs{x} = 3\). We note again that \(D_{p+1}\) contains a unique cyclic subgroup of order 3. As usual, \(N_G(x)\) acts on the set of \(A_4\)s containing \(x\) by conjugation with orbit length at most \(\frac{1}{\abs{x}} (p-1) \eqqcolon k_x\).

	Now, we note that any involutions which lie outside of \(D_{p+1}\) must lie in some shared maximal subgroup with every element of \(A\) of order 3. Thus we have that any such involutions must lie in
	\[\bigcap_{\substack{\abs{x} = 3 \\ x \in A}} \bigcup_{i=1}^{k_x} A_4 \subseteq \bigcup_{i=1}^{k_x} A_4.\]

	We therefore see that 
	\[\abs{A} \leq \frac{p+1}{3} \abs{A_4} + \abs{\Inv D_{p+1}} \leq 4(p+1) + \frac{p+1}{2} + 1 = \frac{9(p+1)}{2} + 1. \qedhere\]
\end{proof}

\begin{lemma} \label{main3b}
	Choose some \(x \in A\) with \(\abs{x} > 2\). If \(\abs{x} \divides p+1\) and \(G\) contains \(A_5\) and \(S_4\) as maximal subgroups then \(\abs{A} \leq \frac{93}{2} (p+1)\).
\end{lemma}

\begin{proof}
	Again, we merge the cases where any of these maximal subgroups appear. As before, if \(\abs{x} > 5\) then we are done by \cref{uniquemax}. Our worst case of course is when both \(A_5\) and \(S_4\) appear as maximal subgroups of \(G\), so this is the one we shall deal with. Here we fix up to three elements, \(x\), \(y\) and \(z\) of respective orders 3, 4 and 5, lying in a single copy of \(D_{p+1}\). Then any element of \(A\) will lie in one of the copies of \(S_4\) or \(A_5\) containing any of \(x\), \(y\) or \(z\) or be an involution in \(D_{p+1}\) since \(D_{p+1}\) has a unique cyclic subgroup of any given order greater than 2. Thus,
	\[A \subseteq \Inv D_{p+1} \cup 2\left( \bigcup_{x \in S_4} S_4 \cup \bigcup_{y \in S_4} S_4 \right) \cup 2 \left( \bigcup_{x \in A_5} A_5 \cup \bigcup_{z \in A_5} A_5 \right).\]
	Using our normal approach to determine how many copies of \(S_4\) or \(A_5\) may contain a given element, noting that their normalisers in these groups are dihedral, we obtain
	\[\abs{A} \leq \frac{p+1}{2} + 120 \left( \frac{p+1}{6} + \frac{p+1}{10} \right) + 48 \left( \frac{p+1}{6} + \frac{p+1}{8} \right) = \frac{93(p+1)}{2}. \qedhere\]
\end{proof}

\begin{lemma} \label{main4a}
	Choose some \(x \in A\) with \(\abs{x} > 2\). If \(\abs{x} \divides p-1\) and \(G\) contains \(A_4\) as a maximal subgroup then \(\abs{A} \leq \frac{17}{2}(p-1) + 6\) (and the \(+6\) may be dropped for \(p > 5\)).
\end{lemma}

\begin{proof}
	Here we need to consider contributions from two Borel subgroups, a copy of \(D_{p-1}\) and several copies of \(A_4\). If \(\abs{x} > 3\) then \(x\) may lie in \(D_{p-1}\) or two Borel subgroups, \(\mathcal{B}_1\) and \(\mathcal{B}_2\). If \(A\) is not one of these subgroups then \(A\) either consists of a few larger order elements along with all of the involutions from the subgroups in which they lie or \(A\) is simply made up of the intersections of these groups. In either case,
	\[A \subseteq (D_{p-1} \cap \mathcal{B}_1 ) \cup (D_{p-1} \cap \mathcal{B}_2) \cup (\mathcal{B}_1 \cap \mathcal{B}_2) \cup \Inv D_{p-1} \cup \Inv \mathcal{B}_1 \cup \Inv \mathcal{B}_2.\]
	But then
	\[\abs{A} \leq \abs{ D_{p-1} \cap \mathcal{B}_1 } + \abs{\Inv D_{p-1}} + 2 \abs{\Inv \mathcal{B}_1} \leq 3p\]
	since \(D_{p-1}\) has at most \(\frac{1}{2}(p+1)\) involutions, and the intersection of the Borel subgroups and \(D_{p-1}\) is the cyclic group of order \(\frac{1}{2}(p-1)\) consisting of those matrices which are diagonal with respect to the basis given by the eigenvectors of the two Borel subgroups.

	Otherwise, we need only concern ourselves with \(\abs{x} = 3\). We note again that \(D_{p-1}\) contains a unique cyclic subgroup of order 3 and also that the Borel subgroups each contain \(p\) distinct cyclic subgroups of order 3. As usual, \(N_G(x)\) acts on the set of \(A_4\)s containing \(x\) by conjugation with orbit length \(\frac{1}{3} (p-1)\).

	If we allow \(nH\) for \(H \leq G\) to denote the union of \(n\) copies of \(H\), we then get
	\[A \subseteq 2p \left< x \right> \cup \Inv \mathcal{B}_1 \cup \Inv \mathcal{B}_2 \cup \Inv D_{p-1} \cup \frac{p-1}{\abs{x}} \Inv A_4\]
	and so, accounting for some double counting,
	\[\abs{A} \leq 4(p-1) + 2p + \frac{p-1}{2} + 3\frac{p-1}{\abs{x}} \leq 4(p-1) + 2p + \frac{p-1}{2} + 3 \frac{p-1}{3} = 15\frac{p-1}{2} + 2\]
	where we have taken \(\frac{1}{2}(p-1)\) in place of \(\abs{\Inv D_{p-1}}\) since the additional involution will be contained in the Borel subgroups if it exists. Alternatively, we have
	\[A \subseteq (\mathcal{B}_1 \cap \mathcal{B}_2 \cap D_{p-1}) \cup 2p \left< x \right> \cup \bigcap_{\substack{\abs{x} = 3 \\ x \in A}} \bigcup_{i=1}^{k_x} A_4\]
	and from the above we may tidy up the final term to get
	\[\abs{A} \leq \frac{p-1}{2} + 4(p-1) + 12\frac{p-1}{\abs{x}} \leq \frac{p-1}{2} + 4(p-1) + 12 \frac{p-1}{3} = 17\frac{p-1}{2}. \qedhere\]
\end{proof}

\begin{lemma} \label{main4b}
	Choose some \(x \in A\) with \(\abs{x} > 2\). If \(\abs{x} \divides p-1\) and \(G\) contains \(A_5\) or \(S_4\) as maximal subgroups then \(\abs{A} \leq \frac{129}{2}(p-1) + 2\).
\end{lemma}

\begin{proof}
	As with the above cases, if \(\abs{x} > 5\) then we are done by \cref{uniquemax}. Otherwise our worst bound comes from when both \(A_5\) and \(S_4\) are present so we consider only this case and proceed in exactly the same way as before to obtain
	\[A \subseteq \Inv (\mathcal{B}_1 \cup \mathcal{B}_2 \cup D_{p-1}) \cup 2p \left< x \right> \cup 2p \left< y \right> \cup 2p \left< z \right> \cup 2 \bigcap_{\substack{\abs{x} \in \{3,4\}\\ x \in A}} \bigcup_{i=1}^{k_x} S_4 \cup \bigcap_{\substack{\abs{x} \in \{3,5\} \\ x \in A}} \bigcup_{i=1}^{k_x} A_5.\]
	Bounding this as usual, and accounting for some of the obvious multiple counting of \(\gen{x}\), \(\gen{y}\) and \(\gen{z}\), we have
	\[\abs{A} \leq \frac{p-1}{2} + 2p + 4(p-1) + 4(p-1) + 8(p-1) + 46(p-1) = \frac{129}{2} (p-1) + 2\]
	as required.
\end{proof}

\begin{lemma} \label{main5}
	Let \(A\) be a maximal coclique in \(G\) containing some element \(x\) of order \(p\). Then \(A\) is a Borel subgroup.
\end{lemma}

\begin{proof}
	If \(\abs{x} > 5\) then this is clear from the orders of the maximal subgroups of \(G\). If \(\abs{x} = 5 = p\) then the only subgroup which could also contain \(x\) is \(A_5\), but this does not occur in this case. If \(\abs{x} = 3\) then \(G = \PSL_2(3) \cong A_4\) and all elements of order 3 lie in unique maximal subgroups. In all cases, we are done by \cref{uniquemax}.
\end{proof}

\begin{proof}[Proof of \cref{main}]
	First note that if \(\left< A \right> \neq G\) then \(A\) is a maximal subgroup by \cref{subgp}. Otherwise, if \(\left< A \right> = G\) then we have several cases to consider. In the first case, \(A\) may consist entirely of involutions since the group generated by two involutions will always be dihedral and thus is never equal to \(G\). Otherwise, either \(A\) contains some element \(x\) of large order such that \(A \setminus \left< x \right>\) consists entirely of involutions or it does not.

	We first assume the former. If \(\abs{x} \divides p-1\) or \(\abs{x} \divides p+1\) then we refer to Lemmas \cref{main1a,main1b,main2a,main2b}. When \(\abs{x} \divides p\) we clearly have that \(\abs{x} = p\) since \(\abs{x} > 2\). But then \(A\) must be a Borel subgroup by the previous lemma.

	We next consider the case where \(A\) contains multiple elements of order greater than 2 which do not lie in some common cyclic group contained in \(A\). The structure of \(A\) is still determined by the orders of its elements. If the orders of the larger elements do not all simultaneously divide \(p-1\), \(p\) or \(p+1\) then we could have elements of orders \(3\) and \(5\) both lying in some copy of \(A_5\), but this case is encompassed in the arguments used in \cref{main3a,main3b,main4a,main4b,main5} which complete the proof.
\end{proof}

\section{The prime-power case}

We now consider the case of \(G \coloneqq \PSL_2(q_0)\) for \(q_0 = p^n\) a prime power. We have little hope of achieving the same result as before due to the existence of copies of \(\PSL_2(r)\) for \(q_0 = r^s\), \(s\) prime, as a maximal subgroup of \(G\). Indeed, simply trying to use the same methods as before would give us bounds to the order of \(q_0^{\frac{3}{2}}\) for all even \(n\) in the case where the coclique was mostly involutions with a single large order element. We also have an interesting geometry which crops up for even \(n\) due to the fact that \(\PSL_2(q^2) \cong \POmega_4^-(q)\) which we will investigate in this section.

At various points throughout this section, the maximal subgroups of \(G\) may be required, so I shall leave them here for reference.  From \cite{HoltMaximals}, as before, we know that the maximal subgroups are the following (where \(q_0 \coloneqq q^2\)):
\begin{enumerate}[i)]
\item The Borel subgroups, \(E_{q^2} \rtimes C_{\frac{1}{2}(q^2 - 1)}\).
\item \(\PSL_2(q).2\), which splits into two conjugacy classes.
\item \(D_{q_0 - 1}\).
\item \(D_{q_0 + 1}\).
\item \(\PSL_2(q_1)\) where \(q = q_1^r\) for an odd prime \(r\).
\item Two conjugacy classes of \(A_5\) for \(q_0 = p^2\), \(p \equiv \pm 3 \mod 10\) prime.
\end{enumerate}
In terms of their classes in Aschbacher's Theorem, we have that a Borel subgroup is the stabiliser of an isotropic 1-space; the \(\PSL_2(q).2\) is the stabiliser of a non-isotropic 1-space; \(D_{q_0 - 1}\) is the stabiliser of a non-degenerate 2-space and the copies of \(\PSL_2(q)\) are stabilisers of subfields of prime index. We will not need to look into the classes of the other types of maximal subgroup.

\subsection{Other cocliques of large order}

We will now calculate explicitly the bound of order \(q_0^{\frac{3}{2}}\) given above, which is evidence that the previous method may not work as well for this situation. We suppose that \(q > 3\) so that all of the required maximal subgroups exist.

As before, we assume that \(A\) is a maximal coclique which contains an element \(x\) with \(\abs{x} > 2\) such that \(A \setminus \left< x \right>\) consists entirely of involutions. If we suppose that \(x\) lies in \(\PSL_2(q)\) and \(\abs{x} \divides q_0-1\) for \(q_0 = q^r,\) \(r\) an odd prime, then at the very least we may include all of the involutions from this subgroup. We then also note that \(N_G(x)\) acts on the set of copies of \(\PSL_2(q)\) containing \(x\) with orbit length
\[\frac{\abs{N_G(x)}}{\abs{N_H(x)}} = \frac{q_0 - 1}{q \pm 1}\]
where \(H \cong \PSL_2(q)\). Including all of the involutions from these other subgroups too, we count
\[\frac{q_0 - 1}{q \pm 1} \abs{\Inv \PSL_2(q)} = \frac{(q_0 - 1)\abs{\PSL_2(q)}}{(q \pm 1)(q \pm 1)} = \frac{q_0 - 1}{q \pm 1} q (q \mp 1)\]
involutions, but there is some multiple-counting which we must account for.

We now suppose that \(\abs{x} = \frac{1}{2}(q + 1)\), so \(x \in D_{q_0 - 1}\). We then note that \(x\) lies in the intersection of all copies of \(\PSL_2(q)\) above and since \(x\) must lie in some maximal subgroup of \(\PSL_2(q)\), we know it must lie in a unique copy of \(D_{q + 1}\). Then we have that the intersection of any two copies of \(\PSL_2(q)\) containing \(x\) must be either \(\left<x\right>\) or \(D_{q + 1}\), but since \(D_{q + 1}\) normalises \(x\) we must also have that \(D_{q + 1} \subseteq D_{q_0-1} = N_G(x)\), and in fact any such \(D_{q + 1}\) must lie in \(N_{D_{q_0-1}}(N_H(x)) \cong D_{2(q + 1)}\). If we assume the worst possible overcounting, then each one of the \(q + 1\) involutions in this subgroup is counted \(\frac{1}{\abs{x}} (q_0 - 1)\) times. Thus we know that our coclique must have at least 
\[2 \frac{q_0-1}{q + 1} q (q + 1) - 2\frac{q_0-1}{q + 1} (q + 1) = 2 \frac{q_0-1}{q + 1} (q^2 - 1)\]
involutions. We don't know if this coclique is necessarily maximal, but we at least have a lower bound on its order.

So we obtain a bound of order \(q_0^{1 + \frac{1}{r}}\) where \(q_0 = q^r\) and \(r\) is the least odd prime such that this happens and the least bound is obtained when we suppose that \(\abs{x} = \frac{1}{2}(q \pm 1)\). It's possible that one could improve this with a good understanding of how these subfield stabilisers intersect, but the smallest case in which this occurs is \(\PSL_2(3^6)\) and is not easy to compute such things in. Doing the same with \(\PSL_2(q).2\) for \(q_0 = q^2\) gives us the bound of order \(q_0^{\frac{3}{2}}\) mentioned above.

We suspect that a similar result will hold to in the prime case and that the geometric anomaly described below is the only exception to what we had before, but the existence of this example makes trying to use the previous methods very complicated. We think that if a maximal coclique \(A\) is such that \(\abs{A} > O(\sqrt{\abs{G}})\) then it either consists of involutions or is a Borel subgroup, if \(\abs{A} = O(\sqrt{\abs{G}})\) then it should be one of the subfield stabiliser maximal subgroups or the geometric example discussed below, and if it is smaller we can't really say much.

\subsection{Construction}

We first let \(q_0 = q^2\) and consider the action of \(G\) on \(\Fq^4\) as \(\POmega_4^-(q)\). Fix some non-isotropic vector \(v\); to ease notation, we let \(v\) denote the \(\Fq\)-span of \(v\) as well as the vector itself since the distinction is not overly important here. Then \(\Fq^4 = v \oplus v^{\perp}\). We wish to consider the elements of \(G\) which have 2-dimensional eigenspaces lying in \(v^{\perp}\), so we must start by determining how many such subspaces one may have.

We wish to collect all of these elements in order to obtain a large coclique with order cubic in \(q\) since if one takes any two elements \(g,\) \(h\) with 2 dimensional eigenspaces \(V_g,\) \(V_h \subseteq v^{\perp}\) then since \(\dim v^{\perp} = 3\) the intersection of any two 2-dimensional subspaces must be nontrivial. Thus \(\left<g,h\right>\) must be contained in the stabiliser of \(V_g \cap V_h \neq \{0\}\) and so \(\left<g,h\right> \neq G\).

A subspace of dimension 3 over \(\Fq\) has \(\frac{q^3-1}{q-1} = q^2 + q + 1\) subspaces of dimension 2, and as seen in \cite{Wilson} we have that if such a subspace (with symmetric bilinear form having gram matrix \(g\)) has an orthonormal basis (\(\det g\) a square) then it is of \(-\) type for \(q \equiv 3 \mod 4\) and if such a basis does not exist (\(\det g\) is a non-square) then it is of \(+\) type. If a space is of neither \(+\) nor \(-\) type then it is degenerate (\(\det g = 0\)). For \(q \equiv 1 \mod 4,\) the 2-spaces with orthonormal bases are of \(+\) type and to obtain a \(-\) type space we must instead consider a space with basis \(\{e,f\}\) where \((e,e) = 1\) and \((f,f) = \alpha\) for some non-square \(\alpha \in \Fgp\).

Using the fact that the orthogonal group acts transitively on isometric subspaces \cite[{}20.8]{Aschbacher} one sees that a 3-dimensional orthogonal \(\Fq\)-space contains \(q \frac{q-1}{2}\) 2-spaces of \(-\) type, \(q \frac{q+1}{2}\) 2-spaces of \(+\) type and \(q+1\) degenerate 2-spaces. We wish to construct a coclique by taking the elements of \(G\) which have any of these 2-spaces as eigenspaces. The only possible eigenvalues for elements of \(G\) are \(\pm 1\).

\begin{lemma}
	Let \(U \subseteq V\) be a degenerate 2-space, then its pointwise stabiliser in \(G\) is isomorphic to \((\mathbb{F}_q,+)\).
\end{lemma}

\begin{proof}
	For this, we refer to \cite[Proposition 2.9.1 (v)]{KleidmanLiebeck} for an explicit isomorphism between \(\PSL_2(q^2)\) and \(\Omega_4^-(q)\). Using the basis \(\{u_1 \coloneqq v_1 \otimes v_1, u_2 \coloneqq v_2 \otimes v_2, w_1 \coloneqq v_1 \otimes v_2 + v_2 \otimes v_2, w_2 \coloneqq \lambda v_1 \otimes v_2 + \bar{\lambda} v_2 \otimes v_1\}\) for \(\bar{\cdot}\) an involutory automorphism of \(\mathbb{F}_{q^2}\) and \(\lambda \in \mathbb{F}_{q^2} \setminus \Fq\) we may now also consider this action from a linear point of view. For one of the \(w_i\), we have that \(\Span_{\Fq} \{ u_1, u_2, w_i\}\) is a 3-space of \(-\) type and contains the degenerate subspace \(W \coloneqq \Span_{\Fq} \{u_1, w_i\}\). The pointwise stabiliser of \(W\) is contained in the stabiliser of its radical, \(\Fq u_1\), and thus lies in a (linear) Borel subgroup. It is then a straightforward calculation using the aforementioned isomorphism to confirm that the pointwise stabiliser of this space is isomorphic to \(E_q\).
\end{proof}

It is a straightforward exercise to check that the sets of elements with given nondegenerate 2-dimensional eigenspaces \(U \subseteq v^{\perp}\) are cyclic of order \(\abs{\SO(U)}\), and one can check this using either the spinor norm (see after \cite[Proposition 2.5.7]{KleidmanLiebeck}) or a calculation similar to that above. In particular, if \(\SO(U) = \gen{h}\) then for \(q \equiv 1 \mod 4\) the sets we are looking for are subgroups generated by \(h^2 \oplus -I_2\) and \(h \oplus -I_2\) for \(U\) of \(+\) type and \(-\) type, respectively. For \(q \equiv 3 \mod 4\) then we instead have \(h \oplus -I_2\) and \(h^2 \oplus -I_2\) (where \(A \oplus B\) is the block-diagonal matrix with \(A\) in the top left and \(B\) in the bottom right).

\begin{lemma}
	The set of all elements with \(2\)-dimensional eigenspaces in \(v^{\perp}\) is a coclique of order \(q^3 + q\).
\end{lemma}

\begin{proof}
	We have a description of the elements with \(2\)-dimensional eigenspaces in \(v^{\perp}\) in the lemmas above, so it is sufficient to check that there is no overlap.

	Suppose that some element \(g\) has 2 distinct 2-dimensional eigenspaces inside of \(v^{\perp}\). Then clearly \(g\) must have a 3-dimensional 1-eigenspace (namely \(v^{\perp}\)) and so \(g\) must also stabilise \(v\) and since we require \(\det g = 1\) we must have that \(g v = v\) and \(g = I_4\). Thus the intersection of any pair of groups generated by the elements of the coclique of maximal order is trivial. Collecting all of these groups together, we see that we obtain a set of size
	\[(q-2) \frac{q (q-1)}{2} + \frac{q^2(q+1)}{2} + (q-1)(q+1) +1 = q^3 + q\]
	and so, as claimed, we obtain a coclique of order cubic in \(q\).
\end{proof}

\begin{thm} \label{final}
	The coclique \(A\) obtained above is maximal.
\end{thm}

We first state a lemma to clean up the end of the proof of this theorem.

\begin{lemma} \label{finallem}
	Suppose \(V\) has a decomposition \(V = V_{\lambda} \oplus V_{\mu} \oplus V_{\nu} \oplus V_{\kappa}\) into distinct (or zero) eigenspaces of \(g \in G\) such that no eigenspace has at least 2-dimensional intersection with \(v^{\perp}\) then there exists some \(h \in A\) such that \(\gen{g, h} = G\).
\end{lemma}

\begin{proof}
	Without loss of generality, we may say that \(V_{\kappa} \cap v^{\perp} = 0\) and the other three eigenspaces have at most 1-dimensional intersection with \(v^{\perp}\). Since an element \(h \in A\) of maximal order fixes some point if and only if it lies in its eigenspace \(U\) (or is one of two isotropic vectors when \(\abs{h} = q-1\)), we may choose any \(h \in A\) of maximal order such that none of the eigenspaces with nontrivial intersection with \(v^{\perp}\) intersect with \(U\) and \(U\) is non-degenerate. We may do this since even in the worst case where \(v^{\perp}\) has a basis of eigenvectors for distinct eigenvalues, there are at most \(3(q+1)\) 2-spaces containing any of these three points, yet \(q^2\) non-degenerate 2-spaces in \(v^{\perp}\) and since we chose \(q > 3\) earlier, \(q^2 > 3(q+1)\). Then \(\left< g, h \right>\) cannot stabilise any proper, nontrivial subspaces of \(V\) and so must be equal to \(G\).
\end{proof}

\begin{proof}[Proof of \cref{final}]
	We now show that the coclique obtained above is indeed maximal by considering the maximal subgroups of \(\POmega_4^-(q)\). In what follows, the term `point' refers to a 1-space.

	This time around, we may ignore \(D_{q^2+1}\) since for \(q > 3\) we have \(q \pm 1 > 2\) and \(r \divides q-1\) and \(r \divides q^2 + 1\) means that \(r \divides (q-1)(q+1) = q^2-1\), thus \(r \divides q^2 + 1 - (q^2 - 1) = 2\). A similar argument holds for \(q+1\), and clearly the only natural number dividing both \(p\) and \(q^2 + 1\) is 1. Then we may ignore \(A_5\) and the subfield stabilisers \(\PSL_2(q)\) by simply choosing the stabilisers of non-degenerate 2-spaces in the argument that follows since these are generated by elements of order \(q \pm 1\) which are not found in \(A_5\) or \(\PSL_2(q)\) for \(q > 3\). if we take an element \(h\) of maximal order in the stabiliser of some non-degenerate subspace and some other \(g \in G\) then we have that \(\left< g,h\right>\) cannot be contained in either \(A_5\) or \(\PSL_2(q)\) as these groups do not have elements of sufficiently large orders.

	We note that the remaining three maximal subgroups are all of class \(\mathcal{C}_1\) and so represent the stabilisers of singular or non-singular subspaces. We thus look at how the stabilisers of various subspaces of \(\Fq^4\) interact with \(A\). We first note that the dihedral subgroups \(D_{q^2 - 1}\) represent the stabilisers of non-degenerate 2-dimensional subspaces (as the stabiliser of a degenerate one lies in the point stabiliser of its radical) and recall that the Borel subgroups and \(\PSL_2(q).2\) are the stabilisers of isotropic and non-isotropic 1-spaces, respectively.

	We also note that an element \(h \in A\) of maximal order stabilising a non-degenerate 2-space \(U \subseteq v^{\perp}\) will stabilise some point if and only if this point lies inside \(U\) or \(\abs{h} = q-1\) and this point is one of two isotropic points in \(U^{\perp}\). We first consider the case where \(g \in G\) stabilises some non-degenerate 2-space. If \(g\) stabilises \(V \subseteq v^{\perp}\) either \(V\) is an eigenspace of \(g\) or there is some 1-space \(\Fq u \subseteq V\) not fixed by \(g\). If there is one such point then there are many others, since if \(g\) fixes more than 2 non-isotropic points in the same 2-space then this space would have to be an eigenspace for \(g\), so we may choose some \(u \in V\) not fixed by \(g\) and a corresponding \(x \in V^{\perp}\) also not fixed by \(g\) such that \(\Span_{\Fq}\{u,x\}\) is non-degenerate. Then the element of \(L\) of maximal order with eigenspace \((\Span_{\Fq} \{u,x\})^{\perp}\) will be such that \(\left<g,h\right>\) stabilises no proper nontrivial subspace of \(\Fq^4\) and so must be equal to \(G\).

	Next, we consider the case where \(g\) stabilises some degenerate 2-space \(V \subseteq v^{\perp}\). Either \(V\) is an eigenspace for \(g\) or there is a (non-isotropic) 1-space \(\Fq u \subseteq V\) not fixed by \(g\). Then, as in the non-degenerate case, we may pick some non-isotropic vector \(x \in V^{\perp}\) such that \(\Span_{\Fq} \{u,x\}\) is non-degenerate and an element \(h \in A\) of maximal order with eigenspace \((\Span_{\Fq} \{u,x\})^{\perp}\) such that \(\left< g, h \right> = G\).

	Otherwise, if \(g\) has a 2-dimensional degenerate eigenspace inside of \(v^{\perp}\) then either the corresponding eigenvalue is 1 and so \(g \in A\) or it is \(-1\) and a direct computation shows us that some power of such an element would be \(-I_4 \nin G\), thus no such element exists.

	We are now left with a number of simpler cases corresponding to the possible eigenspaces of a general element \(g \in G\). If \(g\) has an eigenspace of dimension at least 3 then its intersection with \(v^{\perp}\) must be at least 2-dimensional and so \(g \in A\). Otherwise, \(g\) must have eigenspaces of dimension at most \(2\) and so we know that either it lies in \(A\) or will generate \(G\) along with some element of \(A\) by \cref{finallem}. Maximality of \(A\) follows.
\end{proof}

\newpage

\end{document}